\documentclass[10t]{coulonpaper}

\begin{document}

\setlength{\marginparwidth}{3cm}
\newcommand{\Rcomment}[1]{\Rmodif\marginpar{\tiny\begin{center}\textcolor{blue}{#1}\end{center}}} 
\newcommand{\Vcomment}[1]{\Vmodif\marginpar{\tiny\begin{center}\textcolor{red}{#1}\end{center}}}
\newcommand{\Vmodif}{\ensuremath{\textcolor{red}{\clubsuit}}}
\newcommand{\Rmodif}{$\textcolor{blue}{\spadesuit}$}
\newcommand{\ra}{\rightarrow}
\newcommand{\onto}{\twoheadrightarrow}
\newcommand{\m}{^{-1}}
\newcommand{\grp}[1]{{\langle #1 \rangle}}
\newcommand{\Stab}{\mathrm{Stab}}
\newcommand{\nsub}{\triangleleft}
\newcommand{\semidirect}{\ltimes}
\newcommand{\FR}{\ensuremath{(\mathrm{F}\mathbb{R})}}
\newcommand{\bbR}{\mathbb{R}}

\title{Automorphisms and endomorphisms of lacunary hyperbolic groups}
\author{Rémi Coulon, Vincent Guirardel}

\maketitle
\begin{abstract}
	In this article we study automorphisms and endomorphisms of lacunary hyperbolic groups.
        We prove that every lacunary hyperbolic group is Hopfian, answering a question by Henry Wilton.
	In addition, 
        we show that if a lacunary hyperbolic group has the fixed point property for actions on $\R$-trees, 
        then it is co-Hopfian and
        its outer automorphism group is locally finite. 
        We also construct lacunary hyperbolic groups whose automorphism group
        is infinite, locally finite, and contains any locally finite group given in advance.
\end{abstract}

\tableofcontents

%
\section{Introduction}
%
\label{sec: introduction}

The asymptotic cones of a finitely generated group $G$ are metric spaces introduced by M.~Gromov in his seminal paper on groups with polynomial growth \cite{Gromov:1981ve}.
They are metric spaces obtained as a limit of rescalings of the word metric on $G$, so that the asymptotic cones provide a view of $G$ obtained by ``zooming out to infinity'' on $G$.
A group usually has several asymptotic cones (depending on the choice of rescaling parameters or ultra-filter).
They capture the large scale geometry of $G$.
For instance, a group is hyperbolic if and only if all its asymptotic cones are $\R$-trees \cite{Gromov:1993vu}.
A.Y.~Ol'shanksi\u\i, D.~Osin and M.~Sapir used this idea to define a new notion of negative curvature that extends the class of hyperbolic groups \cite{OlcOsiSap09}.
A finitely generated group $G$ is called \emph{lacunary hyperbolic} if \emph{one} of its asymptotic cone is an $\R$-tree.

\medskip
Alternatively, a lacunary hyperbolic group $G$ can be described algebraically as a direct limit of $\delta_k$-hyperbolic groups $G_k$
with epimorphisms $G_k\ra G$ injective on a ball of radius $R_k$, so that the ratio $R_k/\delta_k$ goes to infinity \cite{OlcOsiSap09}.
This implies in particular that every \emph{finitely presented} lacunary hyperbolic group is actually hyperbolic.
However the class of lacunary hyperbolic groups is much larger: it is an uncountable family of finitely generated groups which are quite diverse and contains some rather exotic infinite groups.
It contains some infinite torsion groups, some non-abelian groups whose proper subgroups are all cyclic, some elementary amenable groups having at least two non-homeomorphic asymptotic cones, and many groups that served as counterexamples to various questions.

\medskip
In this paper we study endomorphisms and automorphisms of lacunary hyperbolic groups.
Although the class of lacunary hyperbolic groups is very large and wild, we manage to bring to light some general properties of their outer automorphism groups.
If $G$ is a hyperbolic group such that $\out G$ is infinite, then the Bestvina-Paulin construction provides a non-trivial action of $G$ on an $\bbR$-tree \cite{Paulin:1991fx}. 
Conversely, this says that if a hyperbolic group $G$ has property $\FR$ (i.e. every action of $G$ on an $\bbR$-tree has a fixed point), then $\out G$ is finite.
This does not hold for general lacunary hyperbolic groups (see below).
However, the following result is reminiscent from this fact.

\begin{mthm}[see \autoref{res: out lacunary locally finite}]
	Let $G$ be a lacunary hyperbolic group with property $\FR$.
	Then $\out G$ is locally finite, i.e. every finitely generated subgroup of $\out G$ is finite.
\end{mthm}

On the other hand, $\out G$ need not be finite.
Moreover, there is no restriction on the possible isomorphism type of locally finite groups occurring in $\out G$ (with the only obvious restriction that it should be countable):

\begin{mthm}[see \autoref{res: embedding locally finite group}]
	Let $A$ be a countable, locally finite group.
	
	Then there exists a lacunary hyperbolic group $G$ with property $\FR$, such that $\out G$ is locally finite 
and contains a subgroup isomorphic to $A$.

	Moreover, given any non-elementary hyperbolic group $H$, one can take $G$ to be a quotient of $H$.
\end{mthm}

\begin{rema}
	In particular, if $A$ is infinite, then $\out G$ is not finitely generated.
\end{rema}

The group $G$ is constructed using iterated small cancellation in hyperbolic groups, where at each step one ensures that
the obtained quotient has more and more automorphisms according to $A$.

\medskip
In general, we exhibit a stratification of $\out G$ related to the fact that $G$ is a controlled direct limit of hyperbolic groups (\autoref{res: lifting automorphisms}).
This has some consequences on the collection of stretching factors of automorphisms of $G$.
Fix a finite generating set $S$ of $G$, and $|\cdot |_S$ the corresponding word metric.
Define the \emph{norm} or \emph{stretching factor} $||\Phi||$ of $\Phi\in \out G$ as
\begin{displaymath}
	\norm\Phi=\inf_{g\in G}\ \max_{s\in S} \ \abs[S]{g \phi(s)g\m}
\end{displaymath}
where $ \phi\in \aut G$ is any representative of $\Phi$.
We then define the \emph{spectrum} $\Lambda\subset \bbR_+^*$ of $\out G$ as the set
of all stretching factors of all elements $\Phi\in\out G$.

\medskip
We prove that if $G$ has property $\FR$\ and is approximated sufficiently fast by hyperbolic groups 
then there are large gaps in the spectrum $\Lambda$.
For instance, if $R_k\geq C \delta_k^\alpha$ for some $\alpha> 2$ with the notations above, 
then there is a sequence of large  intervals $I_i=(a_i,b_i)$ in the complement of $\Lambda$, 
with $a_i$ going to infinity, and whose lengths satisfy $|b_i-a_i|\geq C' a_i^{\alpha-1}$.
Furthermore, if $\alpha>3$, then for all $i$ large enough, 
the set of automorphisms such that $\max \{|\norm\Phi,\norm{\Phi\m}\}\leq a_i$ forms a (finite) subgroup of $\out G$.
See \autoref{sec: spectrum} for more precise statements.

\medskip
Unfortunately we are not able to describe the structure of the outer automorphism group of a lacunary hyperbolic group $G$ without property $\FR$.
One reason is that although we get an action on an $\bbR$-tree, our method does not give any control on arc stabilizers, which prevents us from applying Rips theory.

\bigskip
The other part of this work deals with endomorphisms, and specifically with the Hopf and co-Hopf properties.
Recall that a group $G$ is \emph{Hopfian} if every surjective endomorphism of $G$ is an isomorphism.
Z.~Sela proved that every torsion-free hyperbolic group is Hopfian \cite{Sela:1999ha}.
This result was extended by C.~Reinfeldt and R.~Weidmann to any hyperbolic group \cite{Reinfeldt:xSP4l2Mt}.

\medskip
In \cite{Wilton:2011tp}, Wilton asks about the existence of non-hopfian lacunary hyperbolic groups.
As pointed out by Cornulier in \cite{Wilton:2011tp}, there exists a direct limit of hyperbolic groups which is non-hopfian.
A more elaborate construction by Ivanov-Storozhev yields torsion-free examples that are even free groups in an appropriate variety \cite{Ivanov:2005du}.
Maybe surprisingly, the following result answers Wilton's question by the negative.

\begin{mthm}[see \autoref{res: hopf lacunary}]
	Every lacunary hyperbolic group $G$ is Hopfian.
\end{mthm}

Being Hopfian is a consequence of being {equationally Noetherian}.
Recall that a group $G$ is \emph{equationally Noetherian} if any system of equations
has the same set of solutions as a finite subsystem.
Most of the groups that are known to be Hopfian are in fact equationally Noetherian.
Our result provides example of Hopfian groups that are not equationally Noetherian (see \autoref{res: EN}).


\medskip
Dually, a group $G$ is \emph{co-Hopfian} if every injective endomorphism of $G$ is an isomorphism.
A hyperbolic group is co-Hopfian as soon as it does not split as an amalgamated product or an HNN-extension over a finite group
\cite{Sela:1999ha}, see \cite{Moioli:2013uk} in presence of torsion.
For the class of lacunary hyperbolic groups we have the following statement.

\begin{mthm}[see \autoref{res: cohopf lacunary}]
	Every lacunary hyperbolic group with property $\FR$ is co-Hopfian.
\end{mthm}

Most of the results above rely on the following important observation.
Write $G$ as a direct limit of hyperbolic groups $G_k$ as above.
In any direct limit of finitely presented groups, any endomorphism $\phi\colon G\ra G$ can be lifted 
to a morphism $G_k\ra G_l$ for all $k$ and all $l$ larger than some $l_k$.
Here, we claim that $\phi$ can be lifted to an endomorphism $G_k\ra G_k$ for $k$ large enough.
This is proved using the fact that, being $\delta_k$-hyperbolic, $G_k$ admits a presentation whose relations have length at most $10\delta_k$,
and using that the radius of injectivity of $G_k\ra G$ grows faster than $\delta_k$.

\paragraph{Acknowledgement.}
We would like to warmly thank Henry Wilton for informing us about the question he asked in \cite{Wilton:2011tp}.
The second author acknowledge support from ANR-11-BS01-013 and from the Institut universitaire de France.

%
\section{Lacunary hyperbolic groups.}
%
\label{sec: preliminaries}

Recall that a geodesic metric space is \emph{$\delta$-hyperbolic} if its triangles are $\delta$-thin, i.e. each side is contained in the $\delta$-neighborhood of the union of the other two.
Let $G$  be  a group generated by a finite set $S$.
Then $(G, S)$ is \emph{$\delta$-hyperbolic} if the Cayley graph of $G$ with respect to $S$ is $\delta$-hyperbolic.
A subgroup $H$ of $G$ is called \emph{elementary} if it is virtually cyclic.

\medskip
As we explained in the introduction, A.Y.~Ol'shanksi\u\i, D.~Osin and M.~Sapir defined lacunary hyperbolic groups using asymptotic cones \cite[Section 3.1]{OlcOsiSap09}.
They also provide a complete characterization of such groups in terms of limit of hyperbolic groups.
This is the point of view that we will adopt here.

\begin{defi}
\label{def: injectivity radius}
	Let $G$ be a finitely generated group (endowed with the word metric for a given finite generating set).
	Let $\alpha \colon G \rightarrow G'$ be a homomorphism.
	The \emph{injectivity radius}  of $\alpha$ is the largest integer $\inj\alpha$ such that $\alpha$ restricted to every ball of radius $\inj\alpha$ is one-to-one
\end{defi}

We will write $\alpha:(G,S)\ra(G',S')$ for a morphism $G\ra G'$ sending $S$ to $S'$.

\begin{theo}[{Ol'shanski\u\i-Osin-Sapir \cite[Theorem 3.3]{OlcOsiSap09}}]
\label{res: caracterization lh groups}
	A finitely generated group $G$ is lacunary hyperbolic if and only if for some generating set $S$, one can write $(G,S)$ as the direct limit of a sequence of finitely generated groups $(G_k,S_k)$ and epimorphisms	
	\begin{displaymath}
		(G_0,S_0) \xrightarrow{\alpha_0} (G_1,S_1) \xrightarrow{\alpha_1} \ \cdots\  \ra 
(G_k,S_k) \xrightarrow{\alpha_k}  (G_{k+1},S_{k+1}) \ra 
\ \cdots
	\end{displaymath}
	such that the following holds. 
	For every $k \in \N$, we endow $G_k$  with the word metric with respect to $S_k$, and we denote by $\eta_k\colon(G_k,S_k)\ra (G,S)$ the natural epimorphism.
	Then 
	\begin{enumerate} 
		\item for every $k \in \N$, $(G_k,S_k)$ is $\delta_k$-hyperbolic.
		\item $\displaystyle \lim_{k \rightarrow + \infty} \frac{ \inj{\eta_k} }{\delta_k}= \infty$. 
	\end{enumerate}
\end{theo}

\begin{rema}
  In \cite{OlcOsiSap09}, the statement only says that $\inj{\alpha_k}/\delta_k$ tends to infinity, but the proof really gives the statement above: 
  the groups $G_k$ constructed have a ball of radius $d_k$ that embeds in $G$, and are $\delta_k$-hyperbolic with $d_k/\delta_k$ diverging to infinity. 
  One could also formally deduce the statement above from the one in \cite{OlcOsiSap09}, but the exercise does not seem worth doing.
\end{rema}

%
\section{Structure of the automorphism group}
%
\label{sec: structure out}

Let $G$ be a group generated by a finite set $S$.
We endow $G$ with the word metric $\dist[S]\cdot\cdot$ with respect to $S$, and we denote by
$\abs[S]g=\dist[S]g1$ the corresponding word length.
For every $g \in G$ we write $\iota_g$ for the inner automorphisms that sends $x$ to $gxg^{-1}$.
\begin{defi}
\label{def: norm endomorphism}
	Let $\phi \colon G \rightarrow G$ be an endomorphism of $G$.
	We define two norms  
	\begin{eqnarray*}
		\abs \phi & = & \max_{s \in S}\abs[S]{\phi(s)}\\ 
		\norm \phi & = & \inf_{g \in G} \max_{s \in S} \dist[S]{\phi(s)g}g = \inf_{g \in G} \abs{\iota_{g} \circ \phi}.
	\end{eqnarray*}
\end{defi}

These norms obviously depend on the generating set $S$, but we omit it from the notation as the implicit generating set will be obvious from the context.
Note that for all $g\in G$, we have $\abs[S]{\phi(g)}\leq \abs\phi \cdot \abs[S]g$.

\medskip
We denote by $\smorph G$ the semigroup of all endomorphisms of $G$.
Since balls in $G$ are finite, for all $N\geq 0$, there are only finitely many endomorphisms $\phi$ such that $|\phi|\leq N$.
We define the quotient semi-group $\osmorph G$ of \emph{outer endomorphisms} by identifying two endomophisms $\phi_1,\phi_2$ if there exists $g\in G$ such that $\phi_1 = \iota_g \circ \phi_2$.
Given $\Phi \in \osmorph G$, we define  $\norm \Phi = \norm \phi$ where $\phi\in \smorph G$ is any representative of $\phi$.
For every $N \geq 0$, there are only finitely many $\Phi \in \osmorph G$ such that $\norm \Phi \leq N$.

%
\subsection{Lifting morphisms}
%
\label{sec: lifting}

Let $G$ be a lacunary hyperbolic group.
Using \autoref{res: caracterization lh groups}, we write $G$ as a direct limit of groups 
\begin{displaymath}
	(G_0,S_0) \rightarrow (G_1,S_1) \ra \dots \ra  (G_k,S_k) \ra  (G_{k+1},S_{k+1}) \ra  \dots
\end{displaymath}
where $(G_k,S_k)$ is $\delta_k$-hyperbolic, and 
\begin{displaymath}
	\lim_{ k \rightarrow + \infty} \frac {\inj{\eta_k}}{\delta_k} = \infty.
\end{displaymath}
where we denote by $\pi_k \colon G_0\ra G_k$ and by $\eta_k\colon G_k\ra G$ the natural epimorphisms.
We endow $\smorph G$, $\osmorph G$, $\smorph{G_k}$, $\osmorph{G_k}$ with the norms $|\cdot|$ and $\norm \cdot$ associated with the generating sets $S$ and $S_k$ respectively.

\medskip
Without loss of generality we can assume that $G_0$ is a free group and $S_0$ a basis of $G_0$.
In particular, for any endomorphism $\phi$ of $G$, there exists an endomorphism $\Tilde\phi_0$ of $G_0$ that lifts $\phi$.
Moreover we can choose $\Tilde\phi_0$ such that $|\Tilde\phi_0|=|\phi|$. 
Note however that we cannot assume that $\Tilde \phi_0$ is an automorphism even if $\phi$ is.

\begin{lemm}\label{res: lem_lift}
	Let $\phi\in \smorph G$ and let $\Tilde\phi_0$ be any lift with $|\Tilde\phi_0|=|\phi|$.
	Consider any index $k$ such that $\inj{\eta_k}/ \delta_k > 10 \abs {\phi}$.
	Then $\Tilde\phi_0$ induces an endomorphism $\Tilde\phi_k\in \smorph {G_k}$ that lifts $\phi$, i.e. such that the following diagram commutes:
	\begin{displaymath}
	\label{dia: liftable from step n}
		\begin{tikzpicture}[description/.style={fill=white,inner sep=2pt},baseline=(current bounding box.center)] 
			\matrix (m) [matrix of math nodes, row sep=2em, column sep=2.5em, text height=1.5ex, text depth=0.25ex] 
			{ 
				G_0	& G_k & G	\\
				G_0	& G_k & G	\\
			}; 
			\draw[>=stealth, ->] (m-1-1) -- (m-1-2) node[pos=0.5, above]{$\pi_k$};
			\draw[>=stealth, ->] (m-1-2) -- (m-1-3) node[pos=0.5, above]{$\eta_k$};
			
			\draw[>=stealth, ->] (m-2-1) -- (m-2-2) node[pos=0.5, below]{$\pi_k$};
			\draw[>=stealth, ->] (m-2-2) -- (m-2-3) node[pos=0.5, above]{$\eta_k$};
			
			\draw[>=stealth, ->] (m-1-1) -- (m-2-1) node[pos=0.5, left]{$\Tilde\phi_0$};
			\draw[>=stealth, ->] (m-1-2) -- (m-2-2) node[pos=0.5, left]{$\Tilde\phi_k$};
			\draw[>=stealth, ->] (m-1-3) -- (m-2-3) node[pos=0.5, right]{$\phi$};
		\end{tikzpicture} 
	\end{displaymath}
\end{lemm}

\begin{proof}
	Since $(G_k,S_k)$ is $\delta_k$-hyperbolic, the kernel of $\pi_k \colon G_0\ra G_k$ is normally generated by a finite subset $P_k$ of $G_0$ whose elements have length at most $10\delta_k$ \cite[Chapitre 5, Proposition 1.1]{CooDelPap90}.
	We claim that for all $g\in P_k$, $\pi_k(\Tilde\phi_0(g))=1$.
	The claim  implies that $\Tilde\phi_0$ induces on the quotient an endomorphism $\Tilde \phi_k\in \smorph{G_k}$ which lifts $\phi$ because $\Tilde\phi_0$ does.
	Let us prove the claim.   
	Since $\Tilde\phi_0$ is a lift of $\phi$,  $\Tilde\phi_0(g)$ has trivial image in $G$. This means that $\pi_k(\Tilde\phi_0(g))$ belongs to $\ker\eta_k$.
 	Therefore, to prove that it is trivial, it suffices to prove that $\abs[S_k]{\pi_k(\Tilde\phi_0(g))}<\inj{\eta_k}$.
	Now since $\abs[S_0]g\leq 10\delta_k$ and $\pi_k$ is $1$-Lipschitz, we can conclude that
	\begin{displaymath}
		\abs[S_k]{\pi_k(\Tilde\phi_0(g))}\leq \abs[S_0]{\Tilde \phi_0(g)}\leq \abs{\Tilde\phi_0}\cdot \abs[S_0]g\leq 10\abs{\Tilde\phi_0}\delta_k < \inj{\eta_k}. \qedhere
	\end{displaymath}
\end{proof}

\begin{defi}
\label{def: endom liftable from step n}
	Let $n \in \N$.
	Let $\phi \in \smorph G$.
	We say that $\phi$ \emph{coherently lifts from step $n$} if there exists a lift  $\Tilde \phi_0\in \smorph{G_0}$ of $\phi$ such that for every $k \geq n$, $\Tilde \phi_0$ induces on $G_k$ an endomorphism $\Tilde\phi_k\in\smorph{G_k}$ (necessarily a lift of $\phi$).
	We denote by  $\smorphi nG$ the set of endomorphisms of $G$ which coherently lift from step $n$.
\end{defi} 

\begin{rema}\label{rem_1lip}
This definition implicitly depends on the choice of the approximating sequence given by \autoref{res: caracterization lh groups}.
 We note for future use that if $\Tilde \phi_k\colon G_k\ra G_k$ is any lift of $\phi\colon G\ra G$, then $||\Tilde\phi_k||\geq ||\phi||$.
\end{rema}

By definition, $(\smorphi nG)_{n \in \N}$ is an increasing sequence of semi-groups.
Note that for every $n\in \N$, all  inner automorphisms of $G$ are in $\smorphi nG$.
Hence we can define $\osmorphi nG\subset \osmorph G$ as the quotient of $\smorphi nG$ modulo inner automorphisms.
It follows that $(\osmorphi nG)_{n \in \N}$ is again an increasing sequence of semi-groups.

\begin{coro}
\label{res: lifting automorphisms}
	For every endomorphism $\phi$ of $G$, there exists an integer $n$ such that $\phi$ coherently lifts from step $n$.
	In other words,
	\begin{displaymath}
		\smorph G = \bigcup_{n \in \N} \smorphi nG \quad \text{and} \quad \osmorph G = \bigcup_{n \in \N} \osmorphi nG.
	\end{displaymath}
\end{coro}

\begin{proof}
Since $\inj{\eta_k}/\delta_k$ tends to infinity, there exits $n$ such that for all $k\geq n$, $\inj{\eta_k}/ \delta_k > 10 \abs {\phi}$ so that \autoref{res: lem_lift} applies.
\end{proof}

\subsection{Actions on trees}
\label{sec: actions on trees}

\begin{defi}
\label{def: minimal displacement}
	Let $(X,d)$ be a metric space.
	Let $G$ be a group generated by a finite set $S$ and acting by isometries on $X$.
	We denote by $\lambda(G,X)$ the \emph{minimal displacement} of $G$ on $X$, defined by
	\begin{displaymath}
		\lambda (G,X) = \inf_{x \in X} \max_{s \in S} \dist{sx}x.
	\end{displaymath}
\end{defi}

\begin{defi}
\label{def: property FR}
	A group $G$ has the property \FR, if every action of $G$ on an $\R$-tree has a global fixed point.
\end{defi}

\medskip
For the remainder of this section, $G$ is a lacunary hyperbolic group generated by a finite set $S$ and $(G_k, S_k)$ is the sequence given by \autoref{res: caracterization lh groups}.
In particular $\delta_k$ stands for the hyperbolicity constant of $(G_k, S_k)$. 
The next proposition is a classical variation of the Bestvina-Paulin construction (note that in our statement, the group $G_{k_i}$ acting on the space $X_i$ depends on $i$).

\begin{prop}[\cite{Paulin:1991fx}]
\label{res: bestvina-paulin prelim}
	Consider a sequence of indices $k_i\ra \infty$, and for each $i\geq 0$, consider an isometric action of $G_{k_i}$ on a $1$-hyperbolic space $X_i$ such that $\lim_{i\ra\infty}\lambda(G_{k_i}, X_i)= \infty$.
	Then $G$ acts by isometries without a global fixed point on an $\R$-tree. \qed
\end{prop}

\begin{coro}\label{res: bestvina-paulin for lacunary}
	Assume that $G$ has property $\FR$. 
	Then there exists $k_0$ and a constant $C$ such that for each $k\geq k_0$ and each endomorphism $\phi\in \smorph{G_k}$, one has 
	\begin{displaymath}
		\norm{\phi}\leq C\delta_k.
	\end{displaymath}
\end{coro}

\begin{proof}
  	Assume by contradiction that the corollary does not hold.
	Then there exists a sequence of endomorphisms $\phi_{i}\in \smorph{G_{k_i}}$ where $k_i$ and $\norm{\phi_i}/\delta_{k_i}$ diverge to infinity.
	For each $i\in\N$, denote by $X_{k_i}$ the Cayley graph of $(G_{k_i},S_{k_i})$ endowed with the word metric divided by $\delta_{k_i}$, so that $X_{k_i}$ is $1$-hyperbolic.
	We endow $X_{k_i}$ with a twisted action: for every $g \in G_{k_i}$, for every $x \in X_{k_i}$, $g\cdot x = \phi_i(g)x$.
	Then the minimal displacement of this twisted action is given by
	\begin{displaymath}
		\lambda(G_{k_i}, X_{k_i}) = \frac{1}{\delta_{k_i}}\inf_{x \in X} \max_{s \in S_{k_i}}\dist[S_{k_i}]{\phi_i(s)x}x=\frac{\norm{\phi_i}}{\delta_{k_i}} \ra \infty.
	\end{displaymath}
	By \autoref{res: bestvina-paulin prelim}, we get a contradiction with property $\FR$.
\end{proof}

\begin{coro}
\label{res: bounding displacement}
	Assume $G$ has property \FR, and let  $k_0,C$ be as in \autoref{res: bestvina-paulin for lacunary}. 
	Then for every $n \geq k_0$, for every $\phi \in \smorphi n G$, $\norm \phi \leq C\delta_n $.
\end{coro}

\begin{proof}
	Given $n\geq k_0$ and $\phi\in \smorphi n G$, consider an endomorphism $\Tilde \phi_n:G_n\ra G_n$ lifting $\phi$.
	By \autoref{rem_1lip} and \autoref{res: bestvina-paulin for lacunary}, $||\phi||\leq ||\Tilde \phi_n||\leq C\delta_n$.
\end{proof}

\begin{coro}
\label{res: out lacunary locally finite}
	Assume $G$ has property \FR.
	There exists $A, \kappa \in \R_+$ such that for every $n \in \N$, the cardinality of $\osmorphi nG$ satisfies
	\begin{displaymath}
		\sharp \osmorphi nG \leq Ae^{\kappa\delta_n}.
	\end{displaymath}
	In particular, the semi-group $\osmorph G$ and the group $\out G$ are locally finite.
\end{coro}

\begin{proof}
For any $N\geq 0$, there are at most $(2\sharp S+1)^{N\sharp S}$ endomorphisms $\phi\in\smorph G$, such that $|\phi| \leq N$.
In particular, there are at most $(2\sharp S+1)^{ N\sharp S}$ outer endomorphisms $\Phi\in \osmorph G$, such that $\norm \Phi \leq N$.
The first statement is then a consequence of \autoref{res: bounding displacement}.
Since $\smorph G$ is an increasing union of the finite semigroups $\smorphi nG$, it is locally finite.
Similarly, $\out G$ is locally finite as it is an increasing union of the finite groups $\out G\cap \osmorphi nG$. 
\end{proof}

%
\subsection{Spectrum of the automorphism group}
%
\label{sec: spectrum}

The goal of this section is to study the \emph{spectrum} 
\begin{displaymath}
	\Lambda = \set{\norm\Phi}{\Phi \in \out G}
\end{displaymath}
of norms of outer automorphisms of $G$.
We prove the following proposition telling us that if the hyperbolic groups $G_k$ converge fast enough to $G$, 
then there are arbitrarily large gaps in $\Lambda$.

\begin{prop}
\label{res: large gaps}
	Assume $G$ has property \FR\ and 
	\begin{displaymath}
		\lim_{k \rightarrow + \infty} \frac{\inj{\eta_k}}{\delta_k^2} = \infty.
	\end{displaymath}
	There exists $A,B >0$ and $k_0 \in \N$ such that for every $k \geq k_0$, 
        $\Lambda \cap \left(A \delta_k, B\inj{\eta_k}/\delta_k\right)$ is empty.
\end{prop}

\begin{rema} 
	If $\inj{\eta_k}/\delta_k^2$ is bounded, then the conclusion of the proposition is still true, but vacuous.
\end{rema}

\begin{proof}
	Let $C,k_0$ be given by \autoref{res: bestvina-paulin for lacunary}.
	Let $k\geq k_0$.
	Assume that $\norm \phi < \inj{\eta_k}/10 \delta_k$.
	According to \autoref{res: lem_lift} $\phi$ lifts as an endomorphism $\tilde \phi_k$ of $G_k$.
	Hence by \autoref{res: bestvina-paulin for lacunary} we get $\norm \phi \leq \norm {\tilde \phi_k} \leq C\delta_k$.
	It follows that $\Lambda\cap (C\delta_k,\inj{\eta_k}/10 \delta_k)=\emptyset$.
	The assumption just guarantees that $(C\delta_k,\inj{\eta_k}/10 \delta_k)$ is non-empty for $k$ large enough.
\end{proof}

\begin{coro}
\label{res: subgroup vs displacement}
	Assume $G$ has property \FR\ and
	\begin{displaymath}
		\lim_{k \rightarrow + \infty} \frac{\inj{\eta_k}}{\delta_k^3} = \infty.
	\end{displaymath}
	Then there exists $A>0$ such that for all $k$ large enough, the set 
	\begin{displaymath}
          U_k=\set{\Phi\in \out G}{\max\{\norm \Phi,  \norm{\Phi^{-1}}\} \leq A\delta_k} 
	\end{displaymath}
	is a subgroup of $\out G$.
\end{coro}

\begin{proof}
	Let $A,B >0$ and $k_0 \in \N$ be the constants given by  \autoref{res: large gaps}.
	Up to replacing $k_0$ by a larger integer we can assume that for every $k \geq k_0$, $B \inj{\eta_k}/\delta_k > A^2 \delta_k^2$.
	Consider $k \geq k_0$ and  $\Phi_1, \Phi_2 \in U_k$, and let us prove that $\Phi_1\Phi_2\in U_k$.
	It follows from our choice of $k_0$ that 
	\begin{displaymath}
		\norm{\Phi_1\Phi_2} \leq \norm {\Phi_1} \norm{\Phi_2} \leq A^2 \delta_k^2 < B \frac{\inj{\eta_k}}{\delta_k}.
	\end{displaymath}
	According to \autoref{res: large gaps}, $\norm{\Phi_1 \Phi_2} \leq A \delta_k$, hence $\Phi_1 \Phi_2$ belongs to $U_k$.
\end{proof}

%
\section{Hopf and co-Hopf property}
%
\label{sec: hopf cohopf}

In this section we investigate the Hopf and co-Hopf property for lacunary hyperbolic groups.

\begin{defi}
\label{def: hopf - cohopf}
	A group $G$ is \emph{Hopfian} (\resp \emph{co-Hopfian}) if every endomorphism of $G$ which is surjective (\resp injective) is an automorphism of $G$.
\end{defi}

The Hopf property for torsion-free hyperbolic groups was first proved by Z.~Sela \cite{Sela:1999ha}.
This was extended to hyperbolic groups with torsion by C.~Reinfeldt and R.~Weidmann \cite{Reinfeldt:xSP4l2Mt}
\begin{theo}[{Sela \cite[Theorem~3.3]{Sela:1999ha}, Reinfeldt--Weidmann \cite[Corollary~6.9]{Reinfeldt:xSP4l2Mt}}]
\label{res: hopf hyperbolic}
	Every hyperbolic group is Hopfian.
\end{theo}

Based on this result, we prove the following statement.

\begin{theo}
\label{res: hopf lacunary}
	Every lacunary hyperbolic group is Hopfian.
\end{theo}

\begin{proof}
	Let $G$ be a lacunary hyperbolic group.
	Let 
	\begin{displaymath}
  		(G_0,S_0) \ra (G_1,S_1) \ra  \dots \ra  (G_k,S_k) \ra \dots
	\end{displaymath}
	be the sequence of groups provided by \autoref{res: caracterization lh groups}.
	As above, we assume that $G_0$ is a free group $\free r$ with basis $S_0=\{s_1,\dots,s_r\}$.
 	Let $\phi$ be a surjective endomorphism of $G$.        
	By \autoref{res: lifting automorphisms}, $\phi\in \smorphi{n}G$ for some $n\in \N$, so there exists a lift $\Tilde \phi_0\in \smorph{G_0}$ of $\phi$ that induces lifts $\Tilde \phi_k\in \smorph{G_k}$ for all $k\geq n$.

	\medskip
	We claim that $\Tilde \phi_k$ is surjective for all $k$ large enough. 
	Indeed, $\phi$ being surjective, each element  $s\in S$ can be written as a word on the elements of $\phi(S)$. 
	This fact translates in $G_0$ into the fact that for each $i\leq r$, there is a word $w_i(s_1,\dots,s_r)\in \free r$ such that the element $s_i\m w_i(\Tilde \phi_0(s_1),\dots,\Tilde \phi_0(s_r))\in \free r$ lies in the kernel of $\eta_0\colon \free r\ra G$. 
	As $G$ is the direct limit of $(G_k)$, these elements have trivial image in $G_k$ for $k$ large enough. 
	This implies that $\Tilde \phi_k$ is onto.
          
	\medskip
	Since hyperbolic groups are Hopfian, $\Tilde \phi_k$ is an isomorphism for all $k$ large enough.
	To prove that $\phi$ is one-to-one, consider $g\in \ker\phi$, $g_0\in G_0$ a preimage of $g$, and $g_k=\pi_k(g_0)$ its image in $G_k$.
	Since $\phi(g) = 1$, $\Tilde \phi_0(g_0)\in \ker\eta_0$. 
	However $G$ is the direct limit of $(G_k)$, thus  $\Tilde \phi_k(g_k)=1$ for all $k$ large enough.
	The maps $\Tilde \phi_k$ being ultimately isomorphisms, we get that $g_k = 1$ for $k$ large enough, hence $g=1$.
\end{proof}

The following result gives examples of groups that are Hopfian, but not equationally Noetherian.

\begin{theo}[Champetier {\cite[Théorème~1.3]{Champetier:2000jx}}, Osin \cite{Wilton:2011tp}] 
\label{res: EN}
  There exists a torsion-free lacunary hyperbolic group that is not equationally Noetherian.
\end{theo}

\begin{proof}
  Fix $H_0\xrightarrow{p_0} H_1\xrightarrow{p_1} H_2\onto\cdots$ any sequence of non-injective epimorphisms between torsion-free hyperbolic groups.
We first claim that there exists a lacunary hyperbolic group $G$ such that for all $i \in \N$ and all finite set $F\subset H_i$ there exists a morphism 
$\phi_{i,F}:H_i\to G$ whose restriction to $F$ is one-to-one.
Choose $(i_n,F_n)_{n\in \N}$ an enumeration of all pairs $(i,F)$ where $F$ is a finite subset of $H_i$.
We construct by induction a sequence of torsion-free hyperbolic groups $(G_n,S_n)$ as in \autoref{res: caracterization lh groups} such that for all $k\leq n$,
there exists a morphism $H_{i_k}\to G_n$ whose restriction to $F_k$ is injective.
We start with $G_0=H_{i_0}$ and $S_0$ a finite generating set of $G_0$.

\medskip
Assume now that $(G_n,S_n)$ is constructed. 
Let $B\subset G_n$ be a ball whose radius is sufficiently large compared to the hyperbolicity constant of $(G_n,S_n)$ and containing for all $k\leq n$ the image of $F_k$ by the morphism $H_{i_k}\to G_n$.
Applying \cite[Corollaire~5.21]{Champetier:2000jx}, we can find a common quotient $G_{n+1}$ of $G_n$ and $H_{i_{n+1}}$ such that the projections $G_n\onto G_{n+1}$ and $H_{i_{n+1}}\onto G_{n+1}$ are injective on $B$ and $F_{n+1}$ respectively.
We take for $S_{n+1}$ the image of $S_n$ in $G_{n+1}$. Clearly, $(G_{n+1},S_{n+1})$ satisfies our requirement.
The direct limit $G$ of $(G_n)$ is a torsion-free lacunary hyperbolic group satisfying our claim.

\medskip
Let us check that $G$ is not equationally Noetherian. For all $i\in \N$,
consider $r_i\in \ker p_i \setminus \{1\}$. By construction of $G$, there exists $h_i:H_i\ra G$ which does not kill $r_i$.
In particular $h_i$ does not factor through $p_i$.
Presenting $H_i=\grp{X|R_i}$ with a finite set of relators $R_i\subset R_{i+1}$ so that $p_i$ is induced by the identity on $X$,
the ascending union of all $R_i$ yields a system of equations showing that $G$ is not equationally Noetherian.
\end{proof}

\begin{theo}
\label{res: cohopf lacunary}
	Let $G$ be a lacunary hyperbolic group.
	If $G$ has property \FR\ then $G$ is co-Hopfian.
\end{theo}

\begin{proof}
	Without loss of generality we can assume that $G$ is infinite and not virtually cyclic.
	Otherwise it would act on a line, in contradiction with our hypothesis. 
	Let 
	\begin{equation*}
   (\free r,S_0)  =  (G_0,S_0) \ra (G_1,S_1) \ra  \dots \ra  (G_k,S_k) \ra \dots
	\end{equation*}
	be a sequence of groups provided by \autoref{res: caracterization lh groups}.
 	Let $\phi$ be a monomorphism of $G$. 
	By \autoref{res: lifting automorphisms}, there exists $n$ such that  $\phi\in \smorphi nG$.
	Since $\smorphi nG$ is a semi-group, every non-negative power of $\phi$ belongs to $\smorphi nG$.
	It follows from \autoref{res: bounding displacement} that the  norms $\norm {\phi^m}$ are bounded.
	Since the set of outer endomorphisms of $G$ with bounded norm is finite, there exist $m_1>m_2$ and $a \in G$ such that 
	\begin{equation}
	\label{eqn: cohopf lacunary - powers equal up to conjugacy}
		\phi^{m_1} = \iota_a \circ \phi^{m_2}.
	\end{equation}
	By definition of $\smorphi nG$,  there is a lift $\Tilde \phi_0\in \smorph{G_0}$ of $\phi$ that induces lifts $\Tilde \phi_k\in \smorph{G_k}$ for all $k\geq n$.
	Note that $\Tilde \phi_k$ is not necessarily one-to-one.
	Let $a_0\in G_0$ be a preimage of $a$, and $a_k=\pi_k(a_0)$ its image in $G_k$.
	Since $G$ is finitely generated, $\Tilde \phi_k^{m_1} = \iota_{a_k}\circ \Tilde \phi_k^{m_2}$ for all $k$ larger than some $n'\geq n$.
	
	\medskip
	Fix $k\geq n'$ and $p \in \N$.
	By composing the previous equality on the right and on the left by $\Tilde \phi_k^p$ we get
	\begin{displaymath}
	\label{eqn: cohopf lacunary - centralizing the image}
		\iota_{\Tilde \phi_k^p(a_k)} \circ \Tilde \phi_k^{m_2+p} = \Tilde \phi_k^{m_1 + p} = \iota_{a_k} \circ \Tilde \phi_k^{m_2+ p}.
	\end{displaymath}
	It tells us that the element $u_{k,p}:= a_k^{-1}\Tilde \phi_k^p(a_k)$ belongs to the centralizer of $\Tilde \phi_k^{m_2+p}(G_k)$  
        which we denote by $Z_k(p)$.

	We claim that $Z_k(p)$ is finite.
	Since $G_k$ is hyperbolic, it is sufficient to prove that $\Tilde \phi_k^{m_2+p}(G_k)$ is non elementary.
	Assume on the contrary that $\Tilde \phi_k^{m_2+p}(G_k)$ is virtually cyclic.
	It follows that $\phi^{m_2+p}(G)=\eta_k\circ \Tilde \phi_k^{m_2+p}(G_k)$ in $G$ is virtually cyclic as well.
	Since $\phi$ is one-to-one, $G$ is virtually cyclic, a contradiction. 
	Hence $(Z_k(p))_{p \in \N}$ is an ascending sequence of finite subgroups of $G_k$.
	The group $G_k$ being hyperbolic, their union is finite
	so there exists $p_1 > p_2$ such that $u_{k,p_1} = u_{k,p_2}$, i.e. $\Tilde \phi_k^{p_1}(a_k) = \Tilde \phi_k^{p_2}(a_k)$.
	Pushing this equality in $G$, we get $\phi^{p_1}(a) = \phi^{p_2}(a)$.
	However $\phi$ is one-to-one, thus $a = \phi^r(a)$ with $r = p_1 - p_2$.
	Hence for every $q \in \N^*$, $a = \phi^{qr}(a)$.
	In particular there exists $b \in G$ such that $a = \phi^{m_2}(b)$.
	Consequently (\ref{eqn: cohopf lacunary - powers equal up to conjugacy}) becomes 
	\begin{displaymath}
		\phi^{m_1} = \iota_a \circ \phi^{m_2} = \iota_{\phi^{m_2}(b)} \circ \phi^{m_2} = \phi^{m_2}\circ \iota_b.
	\end{displaymath}
	Since $\phi$ is one-to-one we get $\phi^{m_1-m_2} = \iota_b$, hence $\phi$ is onto.
\end{proof}

%
\section{Embedding locally finite groups}
%
\label{sec: embedding locally finite}

In this section we provide examples of lacunary hyperbolic groups, which are not hyperbolic, and illustrate the results of this paper.
The main goal is to prove the following theorem, showing that every countable, locally finite group can be embedded in the outer automorphism group of a lacunary hyperbolic group.

\begin{theo}
\label{res: embedding locally finite group}
	Let $A$ be a countable, locally finite group, and $G$ be a non-elementary hyperbolic group.
	
	Then there exists a quotient $Q$ of $G$ which is lacunary hyperbolic with property $\FR$ such that $\out Q$ is locally finite and contains a subgroup isomorphic to $A$. 
\end{theo}

The proof of the theorem is based on the following result that will serve as the induction step during the construction.
Following \cite{Olshanskii:1993dr}, given a subgroup $G$ of a group $H$, we define the subgroup $E_H(G)\subset H$ as the set of elements of $H$ whose orbit under conjugation by $G$ is finite. 
Equivalently, $h\in E_H(G)$ if $h$ commutes with a finite index subgroup of $G$.
We note that $E_H(G)\subset E_H(G')$ for $G'\subset G$, with equality if $[G:G']<\infty$.
When $G$ is a non-elementary subgroup of a hyperbolic group $H$, then $E_H(G)$ is the unique maximal finite subgroup of $H$ normalized by $G$ \cite[Proposition~1]{Olshanskii:1993dr}. 

\begin{prop}
\label{res: embedding a larger finite group}
	Consider a non-elementary hyperbolic group $G$, a finite group $A$, and a homomorphism  $\phi \colon A \rightarrow \aut G$ such that
	\begin{nenumerate}[label=(\alph{*}), ref=(\alph{*})]
		\item \label{enu: embedding a larger finite group - hyp1}
		$\phi$ induces an embedding of $A$ into $\out G$; 
		\item \label{enu: embedding a larger finite group - hyp2}
		$E_{\sdp[\phi] GA}(G)=\{1\}$.
	\end{nenumerate} 
	Consider an embedding $j \colon A \rightarrow B$ into a finite group $B$.
	
	\medskip
	Then for every $r \geq 0$, there exists a non-elementary hyperbolic quotient $\bar G$ of $G$ and a homomorphism $\psi \colon B \rightarrow \aut{\bar G}$ with the following properties.
	\begin{enumerate}
		\item The injectivity radius of the canonical projection $G \rightarrow \bar G$ is at least $r$.
		\item \label{enu: embedding a larger finite group - it_cd} For every $a \in A$, the following diagram commutes.
		\begin{center}
			\begin{tikzpicture}[description/.style={fill=white,inner sep=2pt},] 
				\matrix (m) [matrix of math nodes, row sep=2em, column sep=2.5em, text height=1.5ex, text depth=0.25ex] 
				{ 
					G &  \bar G	\\
					G & \bar G	\\
				}; 
				\draw[>=stealth, ->] (m-1-1) -- (m-2-1) node[pos=0.5, left]{$\phi(a)$};
				\draw[>=stealth, ->] (m-1-2) -- (m-2-2) node[pos=0.5, right]{$\psi\circ j(a)$};
				
				\draw[>=stealth, ->] (m-1-1) -- (m-1-2);
				\draw[>=stealth, ->] (m-2-1) -- (m-2-2);
				\end{tikzpicture} 
		\end{center}
		\item \label{enu: embedding a larger finite group - it_ind1} 
		The map $\psi$ induces an embedding of $B$ into $\out{\bar G}$.
		\item \label{enu: embedding a larger finite group - it_ind2} 
        $E_{\sdp[\psi]{\bar G}B}(\bar G)=\{1\}$.
	\end{enumerate}
\end{prop}

Let us first explain how to deduce \autoref{res: embedding locally finite group} from this proposition.

\begin{proof}[Proof of  \autoref{res: embedding locally finite group}] 
	We are going to construct our lacunary hyperbolic group $Q$ as a limit of a sequence of hyperbolic groups, $G_0\onto G_1\onto \dots$.
	We choose for $G_0$ a non-elementary hyperbolic group that is a quotient of $G$, having property $\FR$ and with $E_{G_0}(G_0)=\{1\}$.
	Since property $\FR$ is stable under taking quotients, this will ensure that $Q$ has property $\FR$, and $\out Q$ is locally finite (\autoref{res: out lacunary locally finite}).
	Note that there exists a hyperbolic group $\Gamma$ with property \FR: 
	one can take a triangle group $\Gamma=\grp{x,y|x^2=y^3=(xy)^7}$ by Serre's Lemma \cite[Chapter~4, Lemma~2.1]{Chiswell:2001he}, or a hyperbolic group with Kazhdan's property (T) \cite{BekHarVal08}.
	Then consider $G'_0$, a non-elementary hyperbolic group that is common quotient of $\Gamma$ and $G$.
	This can be achieved by applying \cite[Theorem~2]{Olshanskii:1993dr} to the free product $(G/F) \ast (\Gamma/\Lambda)$ where $F$ and $\Lambda$ stand for the maximal finite normal subgroup of $G$ and $\Gamma$ respectivley; see also \cite[Corollaire~5.21]{Champetier:2000jx}.
	To ensure that $E_{G_0}(G_0)=\{1\}$, take $G_0=G'_0/F'_0$ where $F'_0$ stands for the maximal normal finite subgroup of $G'_0$.

	\medskip
	The group $A$ being countable, we can write $A$ as the union of an increasing sequence of finite groups $(A_k)_{k \in \N}$.
	Without loss of generality, take $A_0=\{1\}$.
	For every $k \in \N$, we denote by $j_k$ the natural embedding $j_k \colon A_k \rightarrow A_{k+1}$.	We choose a finite generating set $S$ for $G$, and we will denote by $S_k$ its image in $G_k$.

	\paragraph{The construction.} 
	Since $A_0$ is the trivial group, we take for $\phi_0$ the trivial morphism, hence $(G_0,A_0,\phi_0)$ obviously satisfies the assumptions \ref{enu: embedding a larger finite group - hyp1}-\ref{enu: embedding a larger finite group - hyp2} of \autoref{res: embedding a larger finite group}.
	
	\medskip
	For the induction step, assume that $(G_k,A_k,\phi_k)$ has been constructed and satisfies the assumptions of \autoref{res: embedding a larger finite group}.
	Let $\delta_k$ be the hyperbolicity constant of $G_k$ with respect to $S_k$.
	If $k \neq 0$, we write $r_{k-1}$ for the injectivity radius of $G_{k-1} \twoheadrightarrow G_k$, otherwise we simply let $r_{k-1} = 0$.
	Apply \autoref{res: embedding a larger finite group} to the embedding $j_k\colon A_k \rightarrow A_{k+1}$, with $r=\max\{r_{k-1},10^k\delta_k\}$. 
	This constructs a quotient $G_{k+1}$ of $G_k$, with a morphism $\phi_{k+1}:A_{k+1}\ra \aut{G_{k+1}}$.
	Since $G_{k+1}$ is a non-elementary hyperbolic group,
	Assertions \ref{enu: embedding a larger finite group - it_ind1},\ref{enu: embedding a larger finite group - it_ind2} of the proposition show that $(G_{k+1},A_{k+1},\phi_{k+1})$
	satisfies the assumptions \ref{enu: embedding a larger finite group - hyp1}-\ref{enu: embedding a larger finite group - hyp2} of \autoref{res: embedding a larger finite group}.
	
	\paragraph{Properties of the limit group.}
	We define $Q$ as the direct limit of the groups $G_k$.
	Since $G_0$ has property \FR, so does $Q$, thus $\out Q$ is locally finite.
	By construction the projection $\pi_k \colon G_k\twoheadrightarrow G_{k+1}$ is injective on the ball of radius $r_k$ and $(r_k)$ is a non-decreasing sequence. 
	Hence the natural morphism $G_k\twoheadrightarrow G$ is injective  on the ball of radius $r_k$.
	On the other hand $r_k/\delta_k$ tends to infinity,
	it follows from \autoref{res: caracterization lh groups} that $Q$ is a lacunary hyperbolic group.

	\medskip
	Let us prove now that $A$ embeds into $\out G$.
	By construction, for each $k$, we have a morphism $\phi_k:A_k\ra \aut {G_k}$.
	Fix $a\in A_{k_0}$, and for $k\geq k_0$, denote by $\alpha_k=\phi_k(a)$ the corresponding automorphism of $G_k$.
	The commutative diagram \ref{enu: embedding a larger finite group - it_cd} of \autoref{res: embedding a larger finite group} shows that $\alpha_{k+1}\circ \pi_k=\pi_k\circ \alpha_{k}$.
	Hence the automorphisms $\alpha_k$ naturally define an automorphism of $Q$ which we denote by $\phi(a)$.
	It is easy to check that the map $\phi \colon A \rightarrow \aut Q$ obtained in this way is a homomorphism.
	There remains to prove that for each $a\in A\setminus \{1\}$, $\phi(a)$ is not inner.
	Assume on the contrary that $\phi(a)=\iota_g$ for some $g\in Q$, and let $g_0\in G_0$ be a preimage of $g$.
	Denote by $g_k$ the image of $g_0$ in $G_k$, and $\alpha_k=\phi_k(a)$ as above.
	Then for all $x_0\in G_0$, denoting by $x_k$ its image in $G_k$, one has $\alpha_k(x_k)=\iota_{g_k}(x_k)$ for all $k$ large enough.
	In particular, there exists $k$ such that $\alpha_k(s)=\iota_{g_k}(s)$ for all $s$ in the finite generating set $ S_k$, so $\phi_k(a)$ is an inner automorphism of $G_k$.
	Since $\phi_k$ induces an embedding of $A_k$ into $\out{G_k}$, we get $a=1$, a contradiction.
\end{proof}

 \autoref{res: embedding a larger finite group} is based on the following variation of a well-known result of A.Y.~Ol'shanski\u\i\ \cite[Theorem~2]{Olshanskii:1993dr}, see also \cite{Champetier:2000jx}.

\begin{theo}
\label{res: small cancellation}
	Let $H$ be a hyperbolic group and $G_1,\dots, G_k$ some non-elementary finitely generated subgroups of $H$.
	Assume that $E_H(G_i)$ is trivial for all $i\in \intvald 1k$.
	Let $N \nsub H$ be a normal subgroup containing all $G_i$'s.
	For every finite subset $P$ of $H$ there exists a quotient $\pi:H\onto \bar H$ with the following properties.
	\begin{enumerate}
		\item \label{enu: small cancellation - ker}
		$\ker\pi\subset N$
		\item \label{enu: small cancellation - hyp}
		 $\bar H$ is a hyperbolic group.
		\item \label{enu: small cancellation - 1:1}
		 The projection $\pi$ restricted to $P$ is one-to-one.
		\item \label{enu: small cancellation - onto}
		 $\pi(G_1)=\pi(G_2)=\dots=\pi(G_k)$ and this group is a non-elementary subgroup of $\bar H$.
		\item \label{enu: small cancellation - prop5} 
		For each finite subgroup $\bar F<\bar H$ there exists a finite subgroup $F<H$ such that $\pi$ maps $F$ isomorphically onto $\bar F$.
	\end{enumerate}
\end{theo}

\begin{proof}
  If $N=H$ and $G_k=H$, Assertions \ref{enu: small cancellation - ker}--\ref{enu: small cancellation - onto} as well as \ref{enu: small cancellation - prop5}   for finite cyclic groups 
are formal consequences of \cite[Theorem~2]{Olshanskii:1993dr}.
In general, \autoref{res: small cancellation} is not a formal consequence of \cite[Theorem~2]{Olshanskii:1993dr} but can be proved in exactly the same way, with minor modifications. 
We explain here the required adjustments.

In the statement of \cite[Theorem~2]{Olshanskii:1993dr}, there is no normal subgroup $N$, and the statement says that one can ensure that $\pi(G_i)=\bar H$.
The proof goes as follows. One adjoins to $H$ some relations of the form $s_\ell=w_{\ell,j}(g_{j,1},\dots,g_{j,k_j})$  where $S=\{s_1,\dots, s_n\}$ is a generating set of $H$, $S_j=\{g_{j,1},\dots, g_{j,k_j}\}$ is a generating set of $G_j$, and $w_{\ell,j}$ a word over the alphabet $S_j$ (see (16) p.~386 in \cite{Olshanskii:1993dr}
for the explicit formula for $w_{\ell,j}$, where $X_{i0}$ in \cite{Olshanskii:1993dr} corresponds to the generator $s_\ell$ in our setting, and $X_{i1},\dots,X_{ik}$
and $W_i$ are in our group $G_i$ as explained p.~403 of \cite{Olshanskii:1993dr}).
The hypothesis $E_H(G_i)=E_H(H)=\{1\}$ is used to prove that this set of relations satisfies a small cancellation assumption (this is \cite[Lemma~4.2]{Olshanskii:1993dr} where a weaker hypothesis is made).
On the other hand, the fact that $S$ generates $H$ is only used to ensure that $\pi(G_i)=\bar H$.

In our setting, we adjoin to $H$ some relations of the form $g_{i,\ell}=w_{i,\ell,j}(g_{j,1},\dots,g_{j,k_j})$.
Our assumption $E_H(G_i)=\{1\}$ guarantees that $w_{i,\ell,j}$ can be chosen in such a way that this set of relations satisfies a small cancellation condition \cite[Lemma~4.2]{Olshanskii:1993dr}. 
The quotient is hyperbolic by \cite[Lemma~6.7]{Olshanskii:1993dr}.
Such relations are obviously in $N$ and ensure that $\pi(G_i)=\pi(G_j)$.

\medskip
Finally, in \cite[Theorem~2]{Olshanskii:1993dr}, Assertion~\ref{enu: small cancellation - prop5} is claimed only for finite cyclic groups.
The argument actually works for all finite groups. 
One can also refer to \cite[Proposition~6.12]{Coulon:2014fr} where this assertion is proved for small cancellation quotients of hyperbolic groups as soon as the adjoint relations are not proper powers which is the case in our setting.
\end{proof}

\begin{proof}[Proof of \autoref{res: embedding a larger finite group}]
If $j$ is surjective, then we can take $\bar G = G$ and $\psi = \phi \circ j^{-1}$.
Hence we may assume that $j(A)$ is a proper subgroup of $B$.
For simplicity we will omit the map $j$ and assume that $A$ is a subgroup of $B$.

\medskip
To apply \autoref{res: small cancellation} above, consider the group $H=(\sdp[\phi] GA) *_A B$.
Since $\sdp[\phi] GA$ contains $G$ with finite index, it is hyperbolic, and so is $H$ by \cite{BesFei92}.
Let $N$ be the normal subgroup of $H$ generated by $G$.
We take for $G_i$ the family of conjugates of $G$ by $B$.
Clearly, the groups $G_i$ are non-elementary. 
Let us check that $E_H(G)=\{1\}$.
This will imply in particular that $E_H(G_i)=\{1\}$ for all $i$.
By definition, some finite index subgroup $G'<G$ commutes with $E_H(G)$.
Since $G'$ fixes a unique point in the Bass-Serre tree of the amalgamated product defining $H$, $E_H(G)$ has to fix this point.
So $E_H(G)$ is a subgroup of $\sdp[\phi] GA$, and $E_H(G)=E_{\sdp[\phi] GA}(G)$, which is trivial by assumption.

\medskip
Let us define the set $P$.
Let $S$ be a generating set of $G$.
Let $P_0$ be the ball of radius $r$ in $G$ (with respect to the word metric relative to $S$).
Let $F_1,\dots, F_k$ be representatives of the conjugacy classes of all finite subgroups of $H$.
For every $i$ we write $P_i$ for the following finite set
\begin{displaymath}
	P_i = \set{sus^{-1}v}{s \in S, \ u,v\in F_i}
\end{displaymath}
We choose for $P$ the union $P=P_0\cup P_1\cup\dots\cup P_k$.

\medskip
We  now apply \autoref{res: small cancellation} and  get a hyperbolic quotient $\pi:H\ra \bar H$ in which all $B$-conjugates of $G$ have the same image.
Define $\bar G$ to be the image of $G$ in $\bar H$.
Since $\bar G$ has finite index in $\bar H$, it is hyperbolic.
Since $\pi$ is injective in restriction to $P_0$, the injectivity radius of the projection $G \rightarrow \bar G$ is at least $r$. 
Since $\bar G$ is normalized by the image of $B$, it is a normal subgroup of $\bar H$.
Let $\psi \colon B \rightarrow \aut {\bar G}$ the morphism corresponding to the action of $B$ on $\bar G$ by conjugation.
Then for each $a\in A$, the diagram in Assertion~\ref{enu: embedding a larger finite group - it_cd} is clearly commutative.
Recall that the kernel of $\pi$ is contained in $N$ (the normal subgroup of $H$ generated by $G$).
Note also that $B$ embeds in $H/N\simeq \bar H/\bar G$. 
In particular, $B$ embeds in $\bar H$ and $B\cap \bar G=\{1\}$,
which implies that $\bar H$ is isomorphic to $\sdp[\psi]{\bar G}{B}$. 

\medskip 
Let us check that $E_{\bar H}(\bar G)=\{1\}$. Since $\bar G$ is non-elementary, $E_{\bar H}(\bar G)$ is a finite subgroup of $\bar H$.
By \autoref{res: small cancellation}~\ref{enu: small cancellation - prop5}, there exists a finite group $F$ that maps isomorphically onto $E_{\bar H}(\bar G)$ under $\pi$. 
Hence there exist $h \in H$ together with an index $i$ such that $F=h^{-1} F_i h$. 
Since $\bar G$ is normal in $\bar H$ we observe that
\begin{displaymath}
	\pi\left(F_i\right)=
	\pi(h)E_{\bar H}\left(\bar G\right) \pi(h)^{-1}
	=E_{\bar H}\left(\pi(h) \bar G \pi(h)^{-1}\right)
	=E_{\bar H}\left(\bar G\right)
\end{displaymath}
Let $u \in F_i$ and $s \in S$.
As we recalled before $E_{\bar H}(\bar G)$ is normalized by $\bar G$.
Hence there exists $v \in F_i$ such that $\pi(sus^{-1}) = \pi(v)$.
It follows that $sus^{-1}v^{-1}$ is an element of $P_i$ in the kernel of $\pi$, thus $sus^{-1} = v$.
Consequently $F_i$ is a finite subgroup of $H$ normalized by $G$.
Our assumption implies that $F_i$ is trivial, and so is $E_{\bar H}(\bar G)$.

 \medskip
There remains to check that $B$ embeds in $\out{\bar G}$.
Let $b \in B$ such that $\psi(b)$ is an inner automorphism.
There exists $g \in G$ such that $\pi( b g)$ centralizes $\bar G$.
In particular, $\pi(bg)$ belongs to $E_{\bar H}(\bar G)$ which is trivial, so $bg\in \ker\pi$. 
Recall that $\ker \pi$ and $G$ are contained in $N$, hence so is $b$. Since $B$ embeds in $H/N$, $b$ has to be trivial.
\end{proof}

\begin{rema}
In the previous construction we have some freedom to choose at each step the injectivity radius of the map $G_k \rightarrow G_{k+1}$.
The same procedure can be used to exhibit a lacunary group such that 
\begin{displaymath}
	\lim_{k \rightarrow + \infty} \frac{r_k}{\delta_k^3} = \infty.
\end{displaymath}
In particular this group will satisfies the assumptions of \autoref{res: large gaps} and \autoref{res: subgroup vs displacement}.
\end{rema}

\begin{rema}
	Elaborating on the ideas of Bumagin and Wise \cite{Bumagin:2005fr} one could probably modify this construction and exhibit a lacunary hyperbolic group $Q$ such that $A = \out Q$.
\end{rema}


\todos

\end{document}